\documentclass[11pt]{amsart}

\usepackage{amsmath,amssymb,latexsym,soul,cite,mathrsfs}
\usepackage{tikz}
\usepackage{color,enumitem,graphicx}
\usepackage[colorlinks=true,urlcolor=blue,
citecolor=red,linkcolor=blue,linktocpage,pdfpagelabels,
bookmarksnumbered,bookmarksopen]{hyperref}
\usepackage[english]{babel}

\usepackage[left=4.3cm,right=4.3cm,top=3.8cm,bottom=3.8cm]{geometry}
\usepackage[hyperpageref]{backref}

\usepackage[colorinlistoftodos]{todonotes}
\makeatletter
\providecommand\@dotsep{5}
\def\listtodoname{List of Todos}
\def\listoftodos{\@starttoc{tdo}\listtodoname}
\makeatother

\usepackage[notref,notcite]{showkeys}

\numberwithin{equation}{section}

\newtheorem{theorem}{Theorem}[section]

\newtheorem{proposition}[theorem]{Proposition}
\newtheorem{lemma}[theorem]{Lemma}

\begin{document}

\title[Fractional Kirchhoff-type systems via sub-supersolutions method in $\mathbb{H}^{\alpha,\beta;\psi}_{p}(\Omega)$]{Fractional Kirchhoff-type systems via sub-supersolutions method in $\mathbb{H}^{\alpha,\beta;\psi}_{p}(\Omega)$}

\author{J. Vanterler da C. Sousa}

\address[J. Vanterler da C. Sousa]
{\newline\indent Aerospace Engineering, PPGEA-UEMA
\newline\indent
Department of Mathematics, DEMATI-UEMA
\newline\indent
São Luís, MA 65054, Brazil.}
\email{\href{vanterler@ime.unicamp.br}{vanterler@ime.unicamp.br}}

\pretolerance10000


\begin{abstract} {\color{red} In the present paper, we first establish a version of the abstract lower and upper-solution method for our class of operators. In this sense, we investigated the main objective of this paper, that is, the existence of a positive solution for a new class of fractional systems of the Kirchhoff type with $\psi$-Hilfer operators via the method of sub and supersolutions in $\psi$-fractional space $\mathbb{H}^{\alpha,\beta;\psi}_{p}(\Omega)$.}

\end{abstract}

\maketitle
\section{Introduction and motivation}

In this paper, we concern a new class of fractional Kirchhoff-type systems given by
\begin{equation}\label{P} 
\left\{
\begin{array}{lcc}
\mathfrak{M}_{1}\left(\displaystyle\int_{\Omega} \left\vert ^{\rm H}\mathfrak{D}^{\alpha,\beta;\psi}_{0+} \eta_{1}\right\vert^{p}d\xi \right) \Delta^{\alpha,\beta;\psi}_{p} \eta_{1} = \zeta a(\xi) f(\eta_{1},\eta_{2}),\,\, in \,\, \Omega=[0,T]\times [0,T]\subset{\mathbb{R}^{2}}\\
    \mathfrak{M}_{2}\left(\displaystyle\int_{\Omega} \left\vert ^{\rm H}\mathfrak{D}^{\alpha,\beta;\psi}_{0+} \eta_{2}\right\vert^{q}d\xi \right) \Delta^{\alpha,\beta;\psi}_{q}\eta_{2} = \zeta b(\xi) \chi(\eta_{1},\eta_{2}),\,\, in \, \Omega=[0,T]\times [0,T]\subset{\mathbb{R}^{2}}\\
\eta_{1}=\eta_{2}=0,\, on\,\partial\Omega
\end{array}
\right.
\end{equation}
where
\begin{equation}\label{fra}
    \Delta^{\alpha,\beta;\psi}_{p} \eta_{1} :=\sum_{i=1}^{2} \,\,
^{\rm H}\mathfrak{D}_{T}^{\alpha ,\beta ;\psi }\left( \left\vert ^{\rm H}\mathfrak{D}_{0+}^{\alpha ,\beta
;\psi }\eta_{1}(\xi_{i})\right\vert ^{p-2}\text{ }^{\rm H}\mathfrak{D}_{0+}^{\alpha ,\beta ;\psi
}\eta_{1}(\xi_{i})\right)
\end{equation}
and
\begin{equation}\label{fra1}
    \Delta^{\alpha,\beta;\psi}_{q} \eta_{2}:=\sum_{i=1}^{2} \,\,
^{\rm H}\mathfrak{D}_{T}^{\alpha ,\beta ;\psi }\left( \left\vert ^{\rm H}\mathfrak{D}_{0+}^{\alpha ,\beta
;\psi }\eta_{2}(\xi_{i})\right\vert ^{q-2}\text{ }^{\rm H}\mathfrak{D}_{0+}^{\alpha ,\beta ;\psi
}\eta_{2}(\xi_{i})\right).
\end{equation}
{\color{red} Here $^{\rm H}\mathfrak{D}^{\alpha,\beta;\psi}_{T} (\cdot)$, $^{\rm H}\mathfrak{D}^{\alpha,\beta;\psi}_{0+} (\cdot)$ are the right and left $\psi$-Hilfer fractional operators of order $\frac{1}{p}<\alpha<1$, $\Omega$ is a bounded smooth domain of $\mathbb{R}^{2}$, $p,q>1$, $p>q$, $\zeta>0$, the functions $\mathfrak{M}_{i},a,b$ satisfy:}

${\bf (H_{1})}$ $\mathfrak{M}_{i}:\mathbb{R}^{+}_{0}\rightarrow\mathbb{R}^{+}$, $i=1,2$, are two continuous and increasing functions and $\mathfrak{M}_{i}(t)\geq m_{i}>0$ for all $t\in \mathbb{R}_{0}^{+}$, where $\mathbb{R}^{+}_{0}=[0,+\infty]$.

${\bf(H_{2})}$ $a,b\in C(\overline{\Omega})$ and $a(\xi)\geq a_{0}>0$, $b(\xi)\geq b_{0}$ for all $\xi\in \overline{\Omega}$.

${\bf(H_{3})}$ $f,\chi \in C^{1} ((0,\infty)\times (0,\infty)) \times C ([0,\infty)\times [0,\infty))$ are monotone functions such that $f_{s}, f_{t}, \chi_{s}, \chi_{t}\geq 0$ and $\lim_{s,t\rightarrow\infty} f(s,t)=\lim_{s,t\rightarrow\infty} \chi(s,t)=\infty$.

${\bf(H_{4})}$ $\lim_{t\rightarrow\infty} \dfrac{f(t,\mathfrak{M}(\chi(t,t))^{1/q-1})}{t^{p-1}}=0$ for $\mathfrak{M}>0$,

${\bf(H_{5})}$ $\lim_{t\rightarrow\infty} \dfrac{\chi(t,t)}{t^{q-1}}=0$.

The Kirchhoff proposed a model given by equation
\begin{equation*}
    \rho \frac{\partial^{2} \eta_{1}}{\partial t^{2}}- \left(\frac{\rho_{0}}{h}+ \frac{E}{2L}\int_{0}^{L} \left\vert \frac{\partial \eta_{1}}{\partial x}\right\vert^{2} dx\right) \frac{\partial^{2} \eta_{1}}{\partial x^{2}}=0,
\end{equation*}
where $\rho, ~\rho_{0}, ~L, ~h,~ E$ are constants, which extends the classical D'Alembert's wave equation.

The operator $$\Delta_{p(x)} u:= -div\left( |\nabla u|^{p(x)-2} |\nabla u| \right)$$ is said to be the $p(x)$-Laplacian, and it becomes $p$-Laplacian when $p(x)=p$. The study of mathematical problems with variable exponents is very interesting. We can highlight the existence and multiplicity problem of the solution of $p(x)$-Laplacian equation, $p(x)$-Kirchhoff and $p$-Kirchhoff both in the classical and in the practical sense \cite{He,Arosio,Correa1,Dai,Dai1,Fan,Fan2,Fan1}. See also the problems involving fractional operators and the references therein \cite{Mingqi,Pucci,Mingqi1,Pucci1}. We can also highlight fractional differential equation problems with $p$-Laplacian using variational methods, in particular, Nehari manifold \cite{Srivastava,Sousa1,Sousa80,Sousa,Sousa2,Sousa3,Ezati,Ezati1}.

In 2006, Correa and Figueiredo \cite{Correa1} considered problems of the $p$-Kirchhoff type
\begin{align*}
\left[-M\left(\int_{\Lambda}\left\vert \nabla \eta_{1}\right\vert ^{p} d x\right)\right]^{p-1} \Delta_{p} \eta_{1}&=f(x,\eta_{1}),\,\, \mbox{in}\,\,\Lambda\notag\\
\eta_{1}&=0,\,\,\mbox{on}\,\,\partial\Lambda,
\end{align*} 
and
\begin{align*}
\left[-M\left(\int_{\Lambda}\left\vert \nabla \eta_{1}\right\vert ^{p} d x\right)\right]^{p-1} \Delta_{p} \eta_{1}&=f(x,\eta_{1})+\zeta |\eta_{1}|^{s-2}\eta_{1},\,\, \mbox{in}\,\,\Lambda\notag\\
\eta_{1}&=0,\,\,\mbox{on}\,\,\partial\Lambda.
\end{align*} 
More details about the problems, see \cite{Correa1}.

In 2010, Fan \cite{Fan2} considered the nonlocal $p(x)$-Laplacian Dirichlet problems
\begin{equation*}
    -A(\eta_{1}) \Delta_{p(x)} \eta_{1}(x)= B(\eta_{1}) f(x,\eta_{1}(x)),\,\,in\,\,\Lambda,\,\,\eta_{1}|_{\partial\Lambda}=0,
\end{equation*}
and 
\begin{align}\label{(222)}
-a\left(\int_{\Lambda}\frac{1}{p(x)}\left\vert \nabla \eta_{1}\right\vert ^{p(x)} d x\right) \Delta_{p(x)} \eta_{1}(x)&=b\left(\int_{\Lambda} F(x,\eta_{1}) \right)f(x,\eta_{1}(x)),\,\, \mbox{in}\,\,\Lambda,\,\,\eta_{1}|_{\partial\Lambda},
\end{align} 
in the non-variational and variational form, respectively, where $$F(x,t)=\displaystyle\int_{0}^{t} f(x,s) ds, $$ and $a$ is allowed to be singular at zero. To obtain the existence and uniqueness of solutions for the problem (\ref{(222)}), the authors used variational methods, especially Mountain pass geometry.

In 2022 Wang et al. \cite{wang122}, established conditions to discuss the existence and uniqueness of global solution for fractional $p$-Kirchhoff equation given by
\begin{equation}\label{kirc}
    \left(a+b\int_{\mathbb{R}^{N}} \int_{\mathbb{R}^{N}} \frac{|\eta_{1}(x)-\eta_{1}(y)|^{p}}{|x-y|^{N+ps}} dxdy \right) (-\Delta)_{p}^{s}\eta_{1}-\mu |\eta_{1}|^{p-2}\eta_{1}= |\eta_{1}|^{q-2}\eta_{1}
\end{equation}
where $f(x,\eta_{1})$ is a general nonlinearity. For more details about the problem (\ref{kirc}), see.

Motivated by the problems highlighted above, in this present paper, our main purpose and contribution is to investigate the existence of a positive solution for fractional problems Kirchhoff-type given by problem (\ref{P}). In other words, we are interested in discussing:

\begin{theorem}\label{Teorema1.1} Under the conditions ${\bf(H_{1})-(H_{5})}$, there exists a positive constant $\zeta_{*}$ such that problem {\rm (\ref{P})} has a positive solution when $\zeta\geq \zeta_{*}$.    
\end{theorem}

A natural consequence when working with fractional operators is, in the limit $\alpha\rightarrow 1$, to obtain the entire case. In this sense, taking the limit $\alpha\rightarrow 1$ and $\psi(\xi)=\xi$, we obtain the classical problem in $\mathbb{R}^{2}$, given by
\begin{equation*} 
\left\{
\begin{array}{lcc}
\mathfrak{M}_{1}\left(\displaystyle\int_{\Omega} \left\vert \dfrac{\partial \eta_{1}}{\partial \xi} \right\vert^{p}d\xi \right) \,\,\displaystyle\sum_{i=1}^{2} \,\,
\dfrac{\partial}{\partial \xi_{i}}\left( \left\vert \dfrac{\partial \eta_{1}}{\partial \xi_{i}}\right\vert ^{p-2}\dfrac{\partial \eta_{1}}{\partial \xi_{i}} \right) = \zeta a(\xi) f(\eta_{1},\eta_{2}),\,\, in \, \Omega\\
    \mathfrak{M}_{2}\left(\displaystyle\int_{\Omega} \left\vert \dfrac{\partial \eta_{2}}{\partial \xi} \right\vert^{q}d\xi \right) \,\,\displaystyle\sum_{i=1}^{2} \,\,
\dfrac{\partial}{\partial \xi_{i}}\left( \left\vert \dfrac{\partial \eta_{2}}{\partial \xi_{i}}\right\vert ^{q-2}\dfrac{\partial \eta_{2}}{\partial \xi_{i}} \right) = \zeta b(\xi) \chi(\eta_{1},\eta_{2}),\,\, in \, \Omega\\
\eta_{1}=\eta_{2}=0,\, on\,\partial\Omega
\end{array}
\right.
\end{equation*}
with the same conditions on $\Theta$ as given in the problem (\ref{P}).
On the other hand, with the freedom of choice of $\beta$ and especially of the function $\psi(\cdot)$, it is possible to obtain a wide class of fractional problems of the Kirchhoff-type. Consequently, the {\bf Theorem \ref{Teorema1.1}} is valid for each particular case. Thus, we finish Section 1.

In Section 2, we present some definitions and results of great importance for the accomplishment of this work. In Section 3, the main result is addressed, i.e., the investigation of the existence of a positive solution to the problem (\ref{P}).

\section{Previous results}

Let $\Omega =(0,T)\times (0,T)\subset \mathbb{R}^{2}$.  Consider the Lebesgue space $L^{p}(\Omega )$ is defined by \cite{Sousa80}
\begin{equation}
L^{p}(\Omega )=\left\{ \Theta(\xi)|\text{ }\Theta\text{ is measurable in }\Omega 
\text{ and }\int_{\Omega }|\Theta|^{p}d\xi<\infty \right\} 
\end{equation}
with the norm
\begin{equation}
\left\Vert \Theta\right\Vert _{L^{p}(\Omega )}=\inf \left\{\int_{\Omega }\left\vert \Theta\right\vert ^{p}d\xi\leq
+\infty \right\} .
\end{equation}

Let $\theta=(\theta_{1},\theta_{2},\theta_{3})$, $T=(T_{1},T_{2},T_{3})$ and $\mu=(\mu_{1},\mu_{2},\mu_{3})$ where $0<\mu_{1},\mu_{2},\mu_{3}<1$ with $\theta_{j}<T_{j}$, for all $j\in \left\{1,2,3 \right\}$. Also put $\Lambda=I_{1}\times I_{2}\times \times I_{3}=[\theta_{1},T_{1}]\times [\theta_{2},T_{2}]\times [\theta_{3},T_{3}]$, where $T_{1},T_{2},T_{3}$ and $\theta_{1},\theta_{2},\theta_{3}$ positive constants. Consider also $\psi(\cdot)$ be an increasing and positive monotone function on $(\theta_{1},T_{1}),(\theta_{2},T_{2}),(\theta_{3},T_{3})$, having a continuous derivative $\psi'(\cdot)$ on $(\theta_{1},T_{1}],(\theta_{2},T_{2}],(\theta_{3},T_{3}]$. The $\psi$-Riemann-Liouville fractional partial integrals of $\Theta\in \mathscr{L}^{1}(\Lambda)$ of order $\mu$ $(0<\mu<1)$ are given by \cite{Srivastava,J1}:
\begin{itemize}
    \item 1-variable: right and left-sided
\begin{equation*}
    {\bf I}^{\mu,\psi}_{\theta_{1}} \Theta(\xi_{1})=\dfrac{1}{\Gamma(\mu_{1})} \int_{\theta_{1}}^{\xi_{1}} \psi'(s_{1})(\psi(\xi_{1})- \psi(s_{1}))^{\mu_{1}-1} \Theta(s_{1}) ds_{1},\,\,to\,\,\theta_{1}<s_{1}<\xi_{1}
\end{equation*}
and
\begin{equation*}
    {\bf I}^{\mu,\psi}_{T_{1}} \Theta(\xi_{1})=\dfrac{1}{\Gamma(\mu_{1})} \int_{\xi_{1}}^{T_{1}} \psi'(s_{1})(\psi(s_{1})- \psi(\xi_{1}))^{\mu_{1}-1} \Theta(s_{1}) ds_{1},\,\,to\,\,\xi_{1}<s_{1}<T_{1},
\end{equation*}
with $\xi_{1}\in[\theta_{1},T_{1}]$, respectively.
    
\item 3-variables: right and left-sided
\begin{eqnarray*}
    &&{\bf I}^{\mu,\psi}_{\theta} \Theta(\xi_{1},\xi_{2},\xi_{3})\notag\\&&=\dfrac{1}{\Gamma(\mu_{1})\Gamma(\mu_{2})\Gamma(\mu_{3})} \int_{\theta_{1}}^{\xi_{1}}
    \int_{\theta_{2}}^{\xi_{2}}
    \int_{\theta_{3}}^{\xi_{3}}
    \psi'(s_{1})\psi'(s_{2})\psi'(s_{3})
    (\psi(\xi_{1})- \psi(s_{1}))^{\mu_{1}-1}\notag\\&&
    \times
    (\psi(\xi_{2})- \psi(s_{2}))^{\mu_{2}-1}
    (\psi(\xi_{3})- \psi(s_{3}))^{\mu_{3}-1}
    \Theta(s_{1},s_{2},s_{3}) ds_{3}ds_{2}ds_{1},
\end{eqnarray*}
to $\theta_{1}<s_{1}<\xi_{1}, \theta_{2}<s_{2}<\xi_{2}, \theta_{3}<s_{3}<\xi_{3}$ and
\begin{eqnarray*}
    &&{\bf I}^{\mu,\psi}_{T} \Theta(\xi_{1},\xi_{2},\xi_{3})\notag\\&&\dfrac{1}{\Gamma(\mu_{1})\Gamma(\mu_{2})\Gamma(\mu_{3})} \int_{\xi_{1}}^{T_{1}}
    \int_{\xi_{2}}^{T_{2}}
    \int_{\xi_{3}}^{T_{3}}
    \psi'(s_{1})\psi'(s_{2})\psi'(s_{3})
    (\psi(s_{1})-\psi(\xi_{1}))^{\mu_{1}-1}\notag\\&&
    \times
    (\psi(s_{2})-\psi(\xi_{2}))^{\mu_{2}-1}
    (\psi(s_{3})-\psi(\xi_{3}))^{\mu_{3}-1}
    \Theta(s_{1},s_{2},s_{3}) ds_{3}ds_{2}ds_{1},
\end{eqnarray*}
with $\xi_{1}<s_{1}<T_{1}, \xi_{2}<s_{2}<T_{2}, \xi_{3}<s_{3}<T_{3}$, $\xi_{1}\in[\theta_{1},T_{1}]$, $\xi_{2}\in[\theta_{2},T_{2}]$ and $\xi_{3}\in[\theta_{3},T_{3}]$, respectively.
\end{itemize}

On the other hand, let $\Theta,\psi \in C^{n}(\Lambda)$ two functions such that $\psi$ is increasing and $\psi'(\xi_{j})\neq 0$ with $\xi_{j}\in[\theta_{j},T_{j}]$, $j\in \left\{1,2,3 \right\}$. The left and
right-sided $\psi$-Hilfer fractional partial derivative of $3$-variables of $\Theta\in AC^{n}(\Lambda)$ of order $\mu=(\mu_{1},\mu_{2},\mu_{3})$ $(0<\mu_{1},\mu_{2},\mu_{3}\leq 1)$ and type $\nu=(\nu_{1},\nu_{2},\nu_{3})$ where $0\leq\nu_{1},\nu_{2},\nu_{3}\leq 1$, are defined by \cite{Srivastava,J1}
\begin{eqnarray}\label{derivada1}
&&{^{\mathbf H}\mathfrak{D}}^{\mu,\nu;\psi}_{\theta}\Theta(\xi_{1},\xi_{2},\xi_{3})\notag\\&&= {\bf I}^{\nu(1-\mu),\psi}_{\theta} \Bigg(\frac{1}{\psi'(\xi_{1})\psi'(\xi_{2})\psi'(\xi_{3})} \Bigg(\frac{\partial^{3}} {\partial \xi_{1}\partial \xi_{2}\partial \xi_{3}}\Bigg) \Bigg) {\bf I}^{(1-\nu)(1-\mu),\psi}_{\theta} \Theta(\xi_{1},\xi_{2},\xi_{3})\notag\\
\end{eqnarray}
and
\begin{eqnarray}\label{derivada2}
&&{^{\mathbf H}\mathfrak{D}}^{\mu,\nu;\psi}_{T}\Theta(\xi_{1},\xi_{2},\xi_{3})\notag\\&&= {\bf I}^{\nu(1-\mu),\psi}_{T} \Bigg(-\frac{1}{\psi'(\xi_{1})\psi'(\xi_{2})\psi'(\xi_{3})} \Bigg(\frac{\partial^{3}} {\partial \xi_{1}\partial \xi_{2}\partial \xi_{3}}\Bigg) \Bigg) {\bf I}^{(1-\nu)(1-\mu),\psi}_{T} \Theta(\xi_{1},\xi_{2},\xi_{3}),\notag\\
\end{eqnarray}
where $\theta$ and $T$ are the same parameters presented in the definition of fractional integrals ${\bf I}_{T}^{\mu;\psi}(\cdot)$ and ${\bf I}_ {\theta}^{\mu;\psi}(\cdot)$. For a study of $N$-variables, see \cite{Srivastava,J1}.

The $\psi$-fractional space is given by \cite{Sousa,Sousa3}
\begin{equation*}
    \mathbb{H}^{\alpha,\beta;\psi}_{r}(\Omega) = \left\{\Theta\in L^{r}(\Omega); \left\vert^{\rm H}\mathfrak{D}^{\alpha,\beta;\psi}_{0+} \Theta \right\vert \in L^{r} (\Omega);\,\ \Theta=0\right\}
\end{equation*}
with the norm
\begin{equation*}
    \left\Vert \Theta \right\Vert = \left\Vert \Theta \right\Vert_{L^{r}}+\left\Vert ^{\rm H}\mathfrak{D}^{\alpha,\beta;\psi}_{0+} \Theta \right\Vert_{L^{r}}.
\end{equation*}
denote by $\mathbb{H}^{\alpha,\beta;\psi}_{r}(\Omega)$ the closure of $C^{\infty}_{0}(\Omega)$ in $\mathbb{H}^{\alpha,\beta;\psi}_{r}(\Omega)$.

Consider the unique solution $e_{r}\in \mathbb{H}^{\alpha,\beta;\psi}_{r}(\Omega)$ of the boundary value problem
\begin{equation}\label{2.3}
^{\rm H}\mathfrak{D}^{\alpha,\beta;\psi}_{T}\left( \left\vert ^{\rm H}\mathfrak{D}^{\alpha,\beta;\psi}_{0+} e_{r} \right\vert^{p}\,\,^{\rm H}\mathfrak{D}^{\alpha,\beta;\psi}_{0+} e_{r} \right) =1,\,\, in\,\, \Omega
\end{equation}
\begin{equation}
    e_{r}=0,\,\,on\,\, \xi\in\partial\Omega,
\end{equation}
to discuss our result.

We say that $(\Theta_{1}, \Theta_{2})$ is a subsolution of problem (\ref{P}) if it is in $\mathbb{H}^{\alpha,\beta;\psi}_{p}(\Omega)\times \mathbb{H}^{\alpha,\beta;\psi}_{q}(\Omega)$ such that
\begin{eqnarray*}
&&\mathfrak{M}_{1}\left(\displaystyle\int_{\Omega} \left\vert ^{\rm H}\mathfrak{D}^{\alpha,\beta;\psi}_{0+} \Theta_{1} \right\vert^{p}d\xi \right)\, \displaystyle\int_{\Omega} \left\vert ^{\rm H}\mathfrak{D}^{\alpha,\beta;\psi}_{0+} \Theta_{1} \right\vert^{p-2}\, ^{\rm H}\mathfrak{D}^{\alpha,\beta;\psi}_{0+} \Theta_{1}\, ^{\rm H}\mathfrak{D}^{\alpha,\beta;\psi}_{0+} w\,d\xi \notag\\
    &&\leq\zeta \displaystyle\int_{\Omega} a(\xi) f(\Theta_{1},\Theta_{2}) w d\xi,\,\, \forall \, w\in W
\end{eqnarray*}
and
\begin{eqnarray*}
&&\mathfrak{M}_{2}\left(\displaystyle\int_{\Omega} \left\vert ^{\rm H}\mathfrak{D}^{\alpha,\beta;\psi}_{0+} \Theta_{2} \right\vert^{q}d\xi \right)\, \displaystyle\int_{\Omega} \left\vert ^{\rm H}\mathfrak{D}^{\alpha,\beta;\psi}_{0+} \Theta_{2} \right\vert^{q-2}\, ^{\rm H}\mathfrak{D}^{\alpha,\beta;\psi}_{0+} \Theta_{2}\, ^{\rm H}\mathfrak{D}^{\alpha,\beta;\psi}_{0+} w\,d\xi \notag\\
    &&\leq\zeta \displaystyle\int_{\Omega} b(\xi) \chi(\Theta_{1},\Theta_{2}) w d\xi,\,\, \forall \, w\in W
\end{eqnarray*}
where $W:=\left\{w\in C^{\infty}_{0}; w\geq 0\,\, in\,\, \Omega \right\}$. On the other hand, we say that $(\Psi_{1},\Psi_{2})\in\mathbb{H}^{\alpha,\beta;\psi}_{p}(\Omega)\times \mathbb{H}^{\alpha,\beta;\psi}_{q}(\Omega)$, is a supersolutions if
\begin{eqnarray*}
&&\mathfrak{M}_{1}\left(\displaystyle\int_{\Omega} \left\vert ^{\rm H}\mathfrak{D}^{\alpha,\beta;\psi}_{0+} \Psi_{1} \right\vert^{p}d\xi \right)\, \displaystyle\int_{\Omega} \left\vert ^{\rm H}\mathfrak{D}^{\alpha,\beta;\psi}_{0+} \Psi_{1} \right\vert^{p-2}\, ^{\rm H}\mathfrak{D}^{\alpha,\beta;\psi}_{0+} \Psi_{1}\, ^{\rm H}\mathfrak{D}^{\alpha,\beta;\psi}_{0+} w\,d\xi \notag\\
    &&\geq\zeta \displaystyle\int_{\Omega} a(\xi) f(\Psi_{1},\Psi_{2}) w d\xi,\,\, \forall \, w\in W
\end{eqnarray*}
and
\begin{eqnarray*}
    &&\mathfrak{M}_{2}\left(\displaystyle\int_{\Omega} \left\vert ^{\rm H}\mathfrak{D}^{\alpha,\beta;\psi}_{0+} \Psi_{2} \right\vert^{q}d\xi \right)\, \displaystyle\int_{\Omega} \left\vert ^{\rm H}\mathfrak{D}^{\alpha,\beta;\psi}_{0+} \Psi_{2} \right\vert^{q-2}\, ^{\rm H}\mathfrak{D}^{\alpha,\beta;\psi}_{0+} \Psi_{2}\, ^{\rm H}\mathfrak{D}^{\alpha,\beta;\psi}_{0+} w\,d\xi \notag\\
    &&\geq\zeta \displaystyle\int_{\Omega} b(\xi) \chi(\Psi_{1},\Psi_{2}) w d\xi,\,\, \forall \, w\in W.
\end{eqnarray*}

\begin{theorem}{\rm\cite{222}} \label{Teorema1.4} Let $\eta_{1},(\eta_{1})_{k}\in E_{p}$, $k=1,2,...$. Then the following statements are equivalent to each other:

1) $\lim_{k\rightarrow\infty}||(\eta_{1})_{k}-\eta_{1}||_{p}=0$;

2) $\lim_{k\rightarrow\infty} p((\eta_{1})_{k}-\eta_{1})=0$;

3) $(\eta_{1})_{k}$ converges to $\eta_{1}$ in $\Omega$ in measure and $\lim_{k\rightarrow\infty} p((\eta_{1})_{k})=p(\eta_{1})$.    
\end{theorem}

{\color{red}The idea behind the proof of {\bf Lemma \ref{Lemma2.1}} is to establish the idea of the supersolutions and subsolutions method for systems, in this case, non-local.}

\begin{lemma}\label{Lemma2.1} Suppose that $\mathfrak{M}:\mathbb{R}^{+}_{0}\rightarrow\mathbb{R}^{+}$ is increasing and continuous. Furthermore, assume that there exists $m_{0}>0$ such that $\mathfrak{M}(t)\geq m_{0}$ for all $t\in \mathbb{R}^{+}_{-}$. If the functions $\eta_{1},\eta_{2}\in \mathbb{H}^{\alpha,\beta;\psi}_{r}(\Omega)$ satisfy
\begin{eqnarray}\label{2.4}
    &&\mathfrak{M}\left(\displaystyle\int_{\Omega} \left\vert ^{\rm H}\mathfrak{D}^{\alpha,\beta;\psi}_{0+} \eta_{1}\right\vert^{r}d\xi \right)\, \displaystyle\int_{\Omega} \left\vert ^{\rm H}\mathfrak{D}^{\alpha,\beta;\psi}_{0+} \eta_{1}\right\vert^{r-2}\, ^{\rm H}\mathfrak{D}^{\alpha,\beta;\psi}_{0+} \eta_{1}\, ^{\rm H}\mathfrak{D}^{\alpha,\beta;\psi}_{0+} \varphi\,d\xi \notag\\
    &&\leq {\color{red}\mathfrak{M}\left(\displaystyle\int_{\Omega} \left\vert ^{\rm H}\mathfrak{D}^{\alpha,\beta;\psi}_{0+} \eta_{2}\right\vert^{r}d\xi \right)\, \displaystyle\int_{\Omega} \left\vert ^{\rm H}\mathfrak{D}^{\alpha,\beta;\psi}_{0+} \eta_{2}\right\vert^{r-2}\, ^{\rm H}\mathfrak{D}^{\alpha,\beta;\psi}_{0+} \eta_{2}\, ^{\rm H}\mathfrak{D}^{\alpha,\beta;\psi}_{0+} \varphi\,d\xi}
\end{eqnarray}
for all $\varphi\in \mathbb{H}^{\alpha,\beta;\psi}_{r}(\Omega)$, $\varphi \geq 0$, then $\eta_{1}\leq \eta_{2}$ in $\Omega$.
\end{lemma}
\begin{proof} Consider the functional $\Xi:\mathbb{H}^{\alpha,\beta;\psi}_{r}(\Omega)\rightarrow\mathbb{R}$ given by
\begin{equation*}
    \Xi(\eta_{1}):=\frac{1}{r} \widehat{\mathfrak{M}} \left(\displaystyle\int_{\Omega} \left\vert ^{\rm H}\mathfrak{D}^{\alpha,\beta;\psi}_{0+} \eta_{2}\right\vert^{r}d\xi\right),\,\, \eta_{1}\in \mathbb{H}^{\alpha,\beta;\psi}_{r}(\Omega).
\end{equation*}
Note that the functional $\Xi$ is a continuously Gateaux differentiable whose Gateaux derivative at the point $\eta_{1}\in \mathbb{H}^{\alpha,\beta;\psi}_{r}(\Omega)$ is the functional $\Xi'\in \mathbb{H}^{\alpha,\beta;\psi}_{r}(\Omega)$ given by
\begin{equation*}
    \Xi'(\eta_{1}) (\varphi) = \mathfrak{M} \left(\displaystyle\int_{\Omega} \left\vert ^{\rm H}\mathfrak{D}^{\alpha,\beta;\psi}_{0+} \eta_{1}\right\vert^{r} d\xi\right) \int_{\Omega} \left\vert ^{\rm H}\mathfrak{D}^{\alpha,\beta;\psi}_{0+} \eta_{1}\right\vert^{r-2}\,\,  ^{\rm H}\mathfrak{D}^{\alpha,\beta;\psi}_{0+} \eta_{1} \, ^{\rm H}\mathfrak{D}^{\alpha,\beta;\psi}_{0+}\varphi d\xi,\,\, 
\end{equation*}
with $\varphi\in \mathbb{H}^{\alpha,\beta;\psi}_{r}(\Omega)$.

Let $\rho_{p}(\eta_{1})=\displaystyle\int_{\Omega} \left\vert ^{\rm H}\mathfrak{D}^{\alpha,\beta;\psi}_{0+} \eta_{1}\right\vert^{r} d\xi$ for all $\eta_{1}\in \mathbb{H}^{\alpha,\beta;\psi}_{r}(\Omega)$. Now, if we choose $(\eta_{1})_{k}=\eta_{2}$ in {\bf Theorem \ref{Teorema1.4}}, we can see that $\eta_{1}=\eta_{2}$ in $\mathbb{H}^{\alpha,\beta;\psi}_{r}(\Omega)$ if and only if $\rho_{r}(\eta_{1})=\rho_{r}(\eta_{2})$. Hence, for any $\eta_{1},\eta_{2}\in \mathbb{H}^{\alpha,\beta;\psi}_{r}(\Omega)$, with $\eta_{1} \neq \eta_{2}$ in  $\mathbb{H}^{\alpha,\beta;\psi}_{r}(\Omega)$, it's follows that $\rho_{r} (\eta_{1}) \neq \rho_{r}(\eta_{2})$. Using the Young inequality, yields
\begin{eqnarray*}
     \displaystyle\int_{\Omega} \left\vert ^{\rm H}\mathfrak{D}^{\alpha,\beta;\psi}_{0+} \eta_{1}\right\vert^{r-2}\, ^{\rm H}\mathfrak{D}^{\alpha,\beta;\psi}_{0+} \eta_{1}\,\, ^{\rm H}\mathfrak{D}^{\alpha,\beta;\psi}_{0+} \eta_{2}\,d\xi \leq \int_{\Omega} \left(\frac{\left\vert ^{\rm H}\mathfrak{D}^{\alpha,\beta;\psi}_{0+} \eta_{2}\right\vert^{r}}{r} +\frac{\left\vert ^{\rm H}\mathfrak{D}^{\alpha,\beta;\psi}_{0+} \eta_{1}\right\vert^{r'}}{r'}\right)d\xi,
\end{eqnarray*}
\begin{eqnarray*}
     \displaystyle\int_{\Omega} \left\vert ^{\rm H}\mathfrak{D}^{\alpha,\beta;\psi}_{0+} \eta_{2}\right\vert^{r-2}\, ^{\rm H}\mathfrak{D}^{\alpha,\beta;\psi}_{0+} \eta_{2}\,\, ^{\rm H}\mathfrak{D}^{\alpha,\beta;\psi}_{0+} \eta_{1}\,d\xi \leq \int_{\Omega} \left(\frac{\left\vert ^{\rm H}\mathfrak{D}^{\alpha,\beta;\psi}_{0+} \eta_{1}\right\vert^{r}}{r} +\frac{\left\vert ^{\rm H}\mathfrak{D}^{\alpha,\beta;\psi}_{0+} \eta_{2}\right\vert^{r'}}{r'}\right)d\xi
\end{eqnarray*}
where $r'=\dfrac{r}{r-1}$. Then, it's follows that
\begin{eqnarray*}
    \left< \Xi'(\eta_{1})-\Xi(\eta_{2}),\eta_{1}-\eta_{2}\right> =&& \left< \Xi'(\eta_{1}),\eta_{1}  \right>- \left<  \Xi'(\eta_{1}),\eta_{2}\right> -\left< \Xi'(\eta_{2}),\eta_{1}  \right>+ \left<  \Xi'(\eta_{2}),\eta_{2}\right> \notag\\ && \geq \left( \mathfrak{M}(\rho(\eta_{1}))-\mathfrak{M}(\rho(\eta_{2}))\right) (\rho(\eta_{1})-\rho(\eta_{2}))\notag\\ && \geq 0,
\end{eqnarray*}
since $\mathfrak{M}(t)$ is monotone. 

{\bf Affirmation: $\Xi'$ is strictly monotone. }

Indeed, if $\left< \Xi'(\eta_{1}),\Xi'(\eta_{2}),\eta_{1}-\eta_{2}  \right>=0$, then we have $\rho_{r}(\eta_{1})=\rho_{r}(\eta_{2})$, i.e., $\eta_{1}=\eta_{2}$, which is contrary to $\eta_{1}\neq \eta_{2}$ in $\mathbb{H}^{\alpha,\beta;\psi}_{r}(\Omega)$. Therefore, $\left< \Xi'(\eta_{1})- \Xi'(\eta_{2}),\eta_{1}-\eta_{2}\right> >0$. Thus, $\Xi'$ is a strictly monotone $\mathbb{H}^{\alpha,\beta;\psi}_{r}(\Omega)$. Now let $\eta_{1},\eta_{2} \in \mathbb{H}^{\alpha,\beta;\psi}_{r}(\Omega)$ such that $(\ref{2.4})$ is verified. Taking $\varphi=(\eta_{1}-\eta_{2})^{+}$, the positive part of $\eta_{1}-\eta_{2}$, as a test function of (\ref{2.4}), yields
\begin{eqnarray}\label{2.5}
&&(\Xi'(\eta_{1})-\Xi(\eta_{2})) (\varphi)\notag\\=&& \mathfrak{M}\left(\displaystyle\int_{\Omega} \left\vert ^{\rm H}\mathfrak{D}^{\alpha,\beta;\psi}_{0+} \eta_{1}\right\vert^{r}d\xi \right)\, \displaystyle\int_{\Omega} \left\vert ^{\rm H}\mathfrak{D}^{\alpha,\beta;\psi}_{0+} \eta_{1}\right\vert^{r-2}\, ^{\rm H}\mathfrak{D}^{\alpha,\beta;\psi}_{0+} \eta_{1}\, ^{\rm H}\mathfrak{D}^{\alpha,\beta;\psi}_{0+} \varphi\,d\xi\notag\\
-&& \mathfrak{M}\left(\displaystyle\int_{\Omega} \left\vert ^{\rm H}\mathfrak{D}^{\alpha,\beta;\psi}_{0+} \eta_{2}\right\vert^{r}d\xi \right)\, \displaystyle\int_{\Omega} \left\vert ^{\rm H}\mathfrak{D}^{\alpha,\beta;\psi}_{0+} \eta_{2}\right\vert^{r-2}\, ^{\rm H}\mathfrak{D}^{\alpha,\beta;\psi}_{0+} \eta_{2}\, ^{\rm H}\mathfrak{D}^{\alpha,\beta;\psi}_{0+} \varphi\,d\xi
\notag\\ &&\leq 0.
\end{eqnarray}

Therefore, the inequality (\ref{2.5}) means that $\eta_{1}\leq \eta_{2}$ in $\Omega$.
\end{proof}

In this sense, consider
\begin{equation}\label{2.6}
\left\{ 
\begin{array}{ccc}
-\mathfrak{M}_{1}\left( \displaystyle\int_{\Omega }\left\vert ^{\rm H}\mathfrak{D}_{0+}^{\alpha ,\beta ;\psi
}\eta_{1}\right\vert ^{p}d\xi\right) \Delta _{p}^{\alpha ,\beta ;\psi }\eta_{1} & = & 
h(\xi,\eta_{1},\eta_{2}),\text{ in }\Omega  \\ 
-\mathfrak{M}_{2}\left(\displaystyle \int_{\Omega }\left\vert ^{\rm H}\mathfrak{D}_{0+}^{\alpha ,\beta ;\psi
}\eta_{2}\right\vert ^{q}d\xi\right) \Delta _{q}^{\alpha ,\beta ;\psi }\eta_{2} & = & 
k(\xi,\eta_{1},\eta_{2}),\text{ in }\Omega  \\ 
\eta_{1}=\eta_{2} & = & 0%
\end{array}%
\right.
\end{equation}
where 
$\Delta _{p}^{\alpha ,\beta ;\psi }(\cdot)$ and $\Delta _{q}^{\alpha ,\beta ;\psi }(\cdot)$ are given by Eq.(\ref{fra}) and Eq.(\ref{fra1}), respectively. Furthermore,
$\Omega =[0,T]\times \lbrack 0,T]$ is a bounded smooth domain of $\mathbb{R}^{2}$ and $h,k:\overline{\Omega }\times \mathbb{R}
\times \mathbb{R}\rightarrow \mathbb{R}$ satisfy:
\begin{itemize}
    \item  {\rm\bf(HK1)} $h(\xi,s,t)$ and $k(\xi,s,t)$ are Caratheodory functions and the are bounded if $s,t$ belongs to bounded sets.

    \item {\rm\bf(HK2)} There exists a function $g:\mathbb{R}\rightarrow \mathbb{R}$ being continuous, nondecreasing with $\chi(0)=0$, $0\leq \chi(s)\leq c(1+|s|^{\min \{p,q\}-1})$ for some $c>0$ and applications $s\rightarrow h(\xi,s,t)+\chi(s)$ and $t\rightarrow k(\xi,s,t)+\chi(t)$ are nondecreasing for a.e. $\xi\in \Omega$.
   \end{itemize}

{\color{red} The present work is the first one that addresses the sub and supersolution technique for problems involving the $\psi$-Hilfer fractional operator. This factor has a great impact on the area, and will certainly allow continuity for future work. The next step will be to present a general lower and upper-solution method. This method has been used by many authors, for example \cite{6,11,26} and the references therein. In this sense, before starting the proof of the main result of this paper, we will make comments in order to establish a Proposition \ref{Proposition2.2} which is of great importance throughout the paper.}

{\color{red} If $\eta_{1},\eta_{2}\in L^{\infty }(\Omega )$ with $\eta_{1}(\xi)\leq \eta_{2}(\xi)$ for a.e. $\xi\in \Omega $, we denote by $[\eta_{1},\eta_{2}]$ the set $\left\{ w\in L^{\infty }(\Omega ):\eta_{1}(\xi)\leq w(\xi)\leq \eta_{2}(\xi)\text{ for q.e. }\xi\in \Omega \right\} $. Using {\bf Lemma \ref{Lemma2.1}} and the method as in the proof of {\bf Theorem 2.4 } \cite{Miyagaki}, we can establish a version of the abstract lower and upper-solution method for our class of the operators as follows.}

\begin{proposition}\label{Proposition2.2} Suppose that $\mathfrak{M}_{1},\mathfrak{M}_{2}:\mathbb{R}_{0}^{+}\rightarrow \mathbb{R} ^{+}$ satisfies the condition ${\rm\bf (H_{1})}$ and $h,k$ satisfy the conditions {\rm\bf (HK1)} and {\rm\bf(HK2)}. Assume that $(\underline{\eta_{1}},\underline{\eta_{2}}),(\overline{\eta_{1}},\overline{\eta_{2}})$ are respectively a weak subsolution and a weak supersolution of system {\rm (\ref{2.6})} with $\underline{\eta_{1}} (\xi)\leq \overline{\eta_{1}}(\xi)$ and $\underline{\eta_{2}}(\xi)\leq \overline{\eta_{2}}(\xi)$ for a.e. 
$\xi\in \Omega .$ Then, there exists a minimal $((\eta_{1})_{\ast },(\eta_{2})_{\ast })$ (and,
respectively a minimal $(\eta_{1}\ast ,\eta_{2}^{\ast })$) weak solution for system {\rm(\ref{2.6})} in the set $[\underline{\eta_{1}},\overline{\eta_{1}}]\times \lbrack \underline{\eta_{2}}, \overline{\eta_{2}}]$. In particular, every weak solution $(\eta_{1},\eta_{2})\in \lbrack \underline{\eta_{1}},\overline{\eta_{1}}]\times \lbrack \underline{\eta_{2}},\overline{\eta_{2}}]$ of system {\rm(\ref{2.6})} satisfies $(\eta_{1})_{\ast }(\xi)\leq \eta_{1}(\xi)\leq \eta_{1}^{\ast }(\xi)$ and $(\eta_{2})_{\ast }(\xi)\leq \eta_{2}(\xi)\leq \eta_{2}^{\ast }(\xi)$ for a.e. $\xi\in \Omega$.    
\end{proposition}

\section{Main results}

In this section, we are going to prove the existence of a solution to the new class of fractional Kirchhoff-type with $p$-Laplacian equation via sub and supersolutions method.

Consider the problem for the fractional $p$-Laplace operator, given by \cite{Sousa1}
\begin{equation}\label{2.1}
\left\{ 
\begin{array}{ccc}
^{\rm H}\mathfrak{D}_{T}^{\alpha ,\beta ;\psi }\left( \left\vert ^{\rm H}\mathfrak{D}_{0+}^{\alpha ,\beta ;\psi }\eta_{1}\right\vert ^{p-2}\,\,^{\rm H}\mathfrak{D}_{0+}^{\alpha ,\beta ;\psi }\eta_{1}\right) 
& = & \zeta |\eta_{1}|^{p-2} \eta_{1},\text{ in }\Omega  \\ 
\eta_{1}& = & 0.
\end{array}%
\right.
\end{equation}

Let $\Theta_{1,p}\in C^{1}(\overline{\Omega})$ be the eigenfunctions corresponding to the first eigenvalues $\zeta_{1,p}$ of problem (\ref{2.1}) $\Theta_{1,p}>0$ in $\Omega$ and $||\Theta_{1,p}||_{\infty}=1$. I can be shown that $\Theta_{1,p}'(0)>0$ on $\partial\Omega$ and hence, depending on $\Omega$, there exists $m,\delta,\sigma>0$ such that
\begin{equation}\label{2.2}
\left\{ 
\begin{array}{ccc}
\left\vert ^{\rm H}\mathfrak{D}_{0+}^{\alpha ,\beta ;\psi }\eta_{1}\right\vert ^{p}-\zeta_{1,p} \Theta_{1,p}^{p} &\geq& m \,\,\,\,\,\,\text{ in }\Omega  \\ 
\Theta_{1,p} & \geq & \sigma\,\,on,\, x\in\Omega/\Omega_{\delta}
\end{array}%
\right.
\end{equation}
where $\overline{\Omega}_{\delta}=\left\{x\in\Omega;d(x,\partial\Omega)\leq \delta \right\}$.

\begin{proof} {\color{red} {\bf (Proof of Theorem \ref{Teorema1.1})}} The idea of the proof is to discuss an extension $f(s,t)$ and $\chi(s,t)$, $\forall(s,t)\in \mathbb{R}^{2}$ such that there exists $k_ {0}>0$ so that $\chi(s,t)\geq \dfrac{-k_{0}}{b_{0}}$ and $f(s,t)\geq - \dfrac {k_{0}}{a_{0}}$, for all $(s,t)\in\mathbb{R}^{2}$. In this sense, consider $\zeta _{1,r},\Theta _{1,r}$ $(r=p,q)$ and $\delta ,m,\sigma ,\Omega _{\delta }$.

We will discuss the proof in two steps.

{\bf Step 1: Subsolution.}

Let's check that it is $\zeta$ large
\begin{equation}
\left( \Phi_{1},\Phi_{2}\right) =\left( \frac{p-1}{p}\left[ \frac{\zeta
k_{0}}{m\mathfrak{M}_{1}}\right] ^{\frac{1}{p-1}}\Theta _{1,p}^{\frac{p}{p-1}},\frac{q-1}{%
q}\left[ \frac{\zeta k_{0}}{m\mathfrak{M}_{2}}\right] \Theta _{1,q}^{\frac{q}{q-1}%
}\right) 
\end{equation}
is a subsolution of problem (\ref{P}). Using the condition
${\bf(H_{1})}$, yields
\begin{eqnarray}\label{3.1}
&&\mathfrak{M}_{1}\left( \int_{\Omega }\left\vert ^{\rm H}\mathfrak{D}_{0+}^{\alpha ,\beta ;\psi }\Phi
_{1}\right\vert ^{p}d\xi\right) \int_{\Omega }\left\vert ^{\rm H}\mathfrak{D}_{0+}^{\alpha
,\beta ;\psi }\Phi_{1}\right\vert ^{p-2}\text{ }^{\rm H}\mathfrak{D}_{0+}^{\alpha ,\beta
;\psi }\Phi_{1}\text{ }^{\rm H}\mathfrak{D}_{0+}^{\alpha ,\beta ;\psi }wd\xi  \notag \\
&=&\frac{\zeta k_{0}}{m\mathfrak{M}_{1}}\mathfrak{M}_{1}\left( \int_{\Omega }\left\vert
^{\rm H}\mathfrak{D}_{0+}^{\alpha ,\beta ;\psi }\Phi_{1}\right\vert ^{p}d\xi\right)\notag\\ &\times&
\int_{\Omega }\left\vert ^{\rm H}\mathfrak{D}_{0+}^{\alpha ,\beta ;\psi }\Theta
_{1,p}\right\vert ^{p-2}\text{ }\Theta _{1,p}\text{ }^{\rm H}\mathfrak{D}_{0+}^{\alpha ,\beta
;\psi }\Theta _{1,p}\text{ }^{\rm H}\mathfrak{D}_{0+}^{\alpha ,\beta ;\psi }wd\xi  \notag \\
&\leq &\frac{\zeta k_{0}}{m\mathfrak{M}_{1}}\mathfrak{M}_{1}\left( \int_{\Omega }\left\vert
^{\rm H}\mathfrak{D}_{0+}^{\alpha ,\beta ;\psi }\Phi_{1}\right\vert ^{p}d\xi\right) \notag \\ &\times &\left[
\int_{\Omega }\left\vert ^{\rm H}\mathfrak{D}_{0+}^{\alpha ,\beta ;\psi }\Theta
_{1,p}\right\vert ^{p-2}\text{ }^{\rm H}\mathfrak{D}_{0+}^{\alpha ,\beta ;\psi }\Theta _{1,p}%
\text{ }^{\rm H}\mathfrak{D}_{0+}^{\alpha ,\beta ;\psi }\Theta _{1,p}\text{ }wd\xi-\int_{\Omega }\left\vert ^{\rm H}\mathfrak{D}_{0+}^{\alpha ,\beta ;\psi }\Theta
_{1,p}\right\vert ^{p-2}\text{ }wd\xi\right]   \notag \\
&=&\frac{\zeta k_{0}}{m\mathfrak{M}_{1}}\mathfrak{M}_{1}\left( \int_{\Omega }\left\vert
^{\rm H}\mathfrak{D}_{0+}^{\alpha ,\beta ;\psi }\Phi_{1}\right\vert ^{p}d\xi\right)
\int_{\Omega }\left( \zeta _{1,p}\Theta _{1,p}^{p}-\left\vert
^{\rm H}\mathfrak{D}_{0+}^{\alpha ,\beta ;\psi }\Theta _{1,p}\right\vert ^{p}\right) wd\xi 
\notag \\
&\leq &\frac{\zeta k_{0}}{m}\int_{\Omega }\left( \zeta _{1,p}\Theta
_{1,p}^{p}-\left\vert ^{\rm H}\mathfrak{D}_{0+}^{\alpha ,\beta ;\psi }\Theta
_{1,p}\right\vert ^{p}\right) wd\xi.\notag\\
\end{eqnarray}

Similarly, we also have
\begin{eqnarray}\label{3.2}
&&{\color{red}\mathfrak{M}_{2}\left( \int_{\Omega }\left\vert ^{\rm H}\mathfrak{D}_{0+}^{\alpha ,\beta ;\psi }\Phi
_{2}\right\vert ^{q}d\xi\right) \int_{\Omega }\left\vert ^{\rm H}\mathfrak{D}_{0+}^{\alpha
,\beta ;\psi }\Phi_{2}\right\vert ^{q-2}\text{ }^{\rm H}\mathfrak{D}_{0+}^{\alpha ,\beta
;\psi }\Phi_{2}\text{ }^{\rm H}\mathfrak{D}_{0+}^{\alpha ,\beta ;\psi }wd\xi } \notag \\
&\leq &\frac{\zeta k_{0}}{m}\int_{\Omega }\left( \zeta _{1,q}\Theta
_{1,q}^{q}-\left\vert ^{\rm H}\mathfrak{D}_{0+}^{\alpha ,\beta ;\psi }\Theta
_{1,q}\right\vert ^{q}\right) wd\xi.\notag\\
\end{eqnarray}

From (\ref{2.2}), yields in $\overline{\Omega _{\delta }}$
\begin{equation}
\zeta _{1,p}\Theta _{1,p}^{p}-\left\vert ^{\rm H}\mathfrak{D}_{0+}^{\alpha ,\beta ;\psi
}\Theta _{1,p}\right\vert ^{p}\leq -m
\end{equation}
and
\begin{equation}
\zeta _{1,q}\Theta _{1,q}^{q}-\left\vert ^{\rm H}\mathfrak{D}_{0+}^{\alpha ,\beta ;\psi
}\Theta _{1,q}\right\vert ^{q}\leq -m.
\end{equation}

{\color{red} Since $f(\Phi_{1},\Phi_{2})\geq -\dfrac{k_{0}}{a_{0}}$, $\chi(\Phi_{1},\Phi_{2})\geq -\dfrac{k_{0}}{b_{0}}$,
$a(\xi)\geq a_{0}>0$, $b(\xi)\geq b_{0}>0$, it follows that in $\overline{\Omega _{\delta }}$ (see to (\ref{2.2}))}
\begin{eqnarray}\label{3.3}
\frac{\zeta k_{0}}{m}\left( \zeta _{1,p}\Theta _{1,p}^{p}-\left\vert
^{\rm H}\mathfrak{D}_{0+}^{\alpha ,\beta ;\psi }\Theta _{1,p}\right\vert ^{p}\right)  &\leq
&\zeta a_{0}f(\Phi_{1},\Phi_{2})  \notag \\
&\leq &\zeta a(\xi)f(\Phi_{1},\Phi_{2})
\end{eqnarray}
and
\begin{eqnarray}\label{3.4}
\frac{\zeta k_{0}}{m}\left( \zeta _{1,q}\Theta _{1,q}^{q}-\left\vert
^{\rm H}\mathfrak{D}_{0+}^{\alpha ,\beta ;\psi }\Theta _{1,q}\right\vert ^{q}\right)  &\leq
&\zeta b_{0}\chi(\Phi_{1},\Phi_{2})  \notag \\
&\leq &\zeta b(\xi)\chi(\Phi_{1},\Phi_{2}).
\end{eqnarray}

On the other hand, in $\Omega \backslash \overline{\Omega _{\delta }}$, we get $\Theta _{1,p}\geq \sigma >0$ and $\Theta _{1,q}\geq \sigma >0$. So, for $\zeta $ large, using the ${\bf(H_{3})}$, (\ref{3.1}) and (\ref{3.2}) that $\forall\xi\in \Omega \backslash \overline{\Omega _{\delta }}$, yields
\begin{eqnarray}\label{*}
\frac{\zeta k_{0}}{m}\left( \zeta _{1,p}\Theta _{1,p}^{p}-\left\vert
^{\rm H}\mathfrak{D}_{0+}^{\alpha ,\beta ;\psi }\Theta _{1,p}\right\vert ^{p}\right)  &\leq &%
\frac{\zeta k_{0}}{m}\zeta _{1,p}  \notag \\
&\leq &\zeta a_{0}f(\Phi_{1},\Phi_{2})  \notag \\
&\leq &\zeta a(\xi)f(\Phi_{1},\Phi_{2})
\end{eqnarray}
and
\begin{eqnarray}\label{**}
\frac{\zeta k_{0}}{m}\left( \zeta _{1,q}\Theta _{1,q}^{q}-\left\vert
^{\rm H}\mathfrak{D}_{0+}^{\alpha ,\beta ;\psi }\Theta _{1,q}\right\vert ^{q}\right)  &\leq &%
\frac{\zeta k_{0}}{m}\zeta _{1,q}  \notag \\
&\leq &\zeta b_{0}\chi(\Phi_{1},\Phi_{2})  \notag \\
&\leq &\zeta b(\xi)\chi(\Phi_{1},\Phi_{2}).
\end{eqnarray}

Hence, replacing (\ref{*}), (\ref{**}) in (\ref{3.1}) and (\ref{3.2}), yields
\begin{eqnarray}
&&\mathfrak{M}_{1}\left( \int_{\Omega }\left\vert ^{\rm H}\mathfrak{D}_{0+}^{\alpha ,\beta ;\psi }\Phi
_{1}\right\vert ^{p}d\xi\right) \int_{\Omega }\left\vert ^{\rm H}\mathfrak{D}_{0+}^{\alpha
,\beta ;\psi }\Phi_{1}\right\vert ^{p-2}\text{ }^{\rm H}\mathfrak{D}_{0+}^{\alpha ,\beta
;\psi }\Phi_{1}\text{ }^{\rm H}\mathfrak{D}_{0+}^{\alpha ,\beta ;\psi }wd\xi  \notag \\
&\leq &\zeta \int_{\Omega }a(\xi)f(\Phi_{1},\Phi_{2})d\xi\notag\\
\end{eqnarray}
and
\begin{eqnarray}
&&\mathfrak{M}_{2}\left( \int_{\Omega }\left\vert ^{\rm H}\mathfrak{D}_{0+}^{\alpha ,\beta ;\psi }\Phi
_{2}\right\vert ^{q}d\xi\right) \int_{\Omega }\left\vert ^{\rm H}\mathfrak{D}_{0+}^{\alpha
,\beta ;\psi }\Phi_{2}\right\vert ^{q-2}\text{ }^{\rm H}\mathfrak{D}_{0+}^{\alpha ,\beta
;\psi }\Phi_{2}\text{ }^{\rm H}\mathfrak{D}_{0+}^{\alpha ,\beta ;\psi }wd\xi  \notag \\
&\leq &\zeta \int_{\Omega }b(\xi)\chi(\Phi_{1},\Phi_{2})d\xi.\notag\\
\end{eqnarray}

Thus, we have a subsolution $(\Phi_{1},\Phi_{2})$ of problem (\ref{P}) for $\zeta $ large.

{\bf Step II: Supersolution:}

Let
\begin{equation}
\Psi_{1}=\frac{c}{l_{p}}\left( \zeta \left\Vert a\right\Vert _{\infty
}\right) ^{\frac{1}{p-1}}e_{p},\text{ }\Psi_{2}=\left( \chi\left( c\zeta ^{\frac{%
1}{p-1}},c\zeta ^{\frac{1}{p-1}}\right) \right) ^{\frac{1}{q-1}}\left( 
\frac{\zeta \left\Vert b\right\Vert _{\infty }}{\mathfrak{M}_{2}}\right) ^{\frac{1}{%
q-1}}e_{q}
\end{equation}
where $e_{p},e_{q}$ are defined by (\ref{2.3}) and $l_{r}=\left\Vert
e_{r}\right\Vert _{\infty }$ for $r=p,q$. The objective is to prove that $(\Psi_{1},\Psi_{2})$ is a supersolution of problem (\ref{P}). So, we get
\begin{eqnarray}\label{3.5}
&&\mathfrak{M}_{1}\left( \int_{\Omega }\left\vert ^{\rm H}\mathfrak{D}_{0+}^{\alpha ,\beta ;\psi
}\Psi_{1}\right\vert ^{p}d\xi\right) \int_{\Omega }\left\vert ^{\rm H}\mathfrak{D}_{0+}^{\alpha
,\beta ;\psi }\Psi_{1}\right\vert ^{p-2}\text{ }^{\rm H}\mathfrak{D}_{0+}^{\alpha ,\beta ;\psi
}\Psi_{1}\text{ }^{\rm H}\mathfrak{D}_{0+}^{\alpha ,\beta ;\psi }wd\xi  \notag \\
&=&\frac{\left\Vert a\right\Vert _{\infty }}{\left( l_{p}\right) ^{p-1}}%
\left( \zeta ^{\frac{1}{p-1}}c\right) ^{p-1}\mathfrak{M}_{1}\left( \int_{\Omega
}\left\vert ^{\rm H}\mathfrak{D}_{0+}^{\alpha ,\beta ;\psi }\Psi_{1}\right\vert ^{p}d\xi\right) \notag\\&\times&
\text{ }\int_{\Omega }\left\vert ^{\rm H}\mathfrak{D}_{0+}^{\alpha ,\beta ;\psi
}e_{p}\right\vert ^{p-2}\text{ }^{\rm H}\mathfrak{D}_{0+}^{\alpha ,\beta ;\psi }e_{p}\text{ 
}^{\rm H}\mathfrak{D}_{0+}^{\alpha ,\beta ;\psi }wd\xi  \notag \\
&\geq &\frac{\mathfrak{M}_{1}\left\Vert a\right\Vert _{\infty }}{\left( l_{p}\right)
^{p-1}}\left( \zeta ^{\frac{1}{p-1}}c\right) ^{p-1}\int_{\Omega }wd\xi.\notag\\
\end{eqnarray}

Using the condition ${\bf(H_{4})}$, take $c$ large enough such that
\begin{eqnarray}\label{3.6}
&&\frac{\mathfrak{M}_{1}}{\left( l_{p}\right) ^{p-1}}\left( \zeta ^{\frac{1}{p-1}
}c\right) ^{p-1}\int_{\Omega }wd\xi \notag\\&\geq &\zeta \int_{\Omega }f\left(
\zeta ^{\frac{1}{p-1}}c,\left[ \chi(\zeta ^{\frac{1}{p-1}}c,\zeta ^{\frac{%
1}{p-1}}c)\right] ^{\frac{1}{q-1}}\left( \frac{\zeta \left\Vert
b\right\Vert _{\infty }}{\mathfrak{M}_{2}}\right) ^{\frac{1}{q-1}}l_{q}\right) wd\xi 
\notag \\
&\geq &\zeta \int_{\Omega }f\left( \zeta ^{\frac{1}{p-1}}c\frac{e_{p}}{%
l_{p}},\left[ \chi(\zeta ^{\frac{1}{p-1}}c,\zeta ^{\frac{1}{p-1}}c)\right]
^{\frac{1}{q-1}}\left( \frac{\zeta \left\Vert b\right\Vert _{\infty }}{%
\mathfrak{M}_{2}}\right) ^{\frac{1}{q-1}}e_{q}\right) wd\xi  \notag \\
&\geq &\zeta \int_{\Omega }f\left( \Psi_{1},\Psi_{2}\right) d\xi.
\end{eqnarray}

From (\ref{3.5}) and (\ref{3.6}), yields
\begin{eqnarray}\label{3.7}
&&\mathfrak{M}_{1}\left( \int_{\Omega }\left\vert ^{\rm H}\mathfrak{D}_{0+}^{\alpha ,\beta ;\psi
}\Psi_{1}\right\vert ^{p}d\xi\right) \int_{\Omega }\left\vert ^{\rm H}\mathfrak{D}_{0+}^{\alpha
,\beta ;\psi }\Psi_{1}\right\vert ^{p-2}\text{ }^{\rm H}\mathfrak{D}_{0+}^{\alpha ,\beta ;\psi
}\Psi_{1}\text{ }^{\rm H}\mathfrak{D}_{0+}^{\alpha ,\beta ;\psi }wd\xi  \notag \\
&\geq &\zeta \left\Vert a\right\Vert _{\infty }\int_{\Omega }f\left(
\Psi_{1},\Psi_{2}\right) d\xi  \notag \\
&\geq &\zeta a(\xi)\int_{\Omega }f\left( \Psi_{1},\Psi_{2}\right) d\xi.\notag\\
\end{eqnarray}

On the other hand, we have
\begin{eqnarray}\label{3.8}
&&\mathfrak{M}_{2}\left( \int_{\Omega }\left\vert ^{\rm H}\mathfrak{D}_{0+}^{\alpha ,\beta ;\psi
}\Psi_{2}\right\vert ^{p}d\xi\right) \int_{\Omega }\left\vert ^{\rm H}\mathfrak{D}_{0+}^{\alpha
,\beta ;\psi }\Psi_{2}\right\vert ^{p-2}\text{ }^{\rm H}\mathfrak{D}_{0+}^{\alpha ,\beta ;\psi
}\Psi_{2}\text{ }^{\rm H}\mathfrak{D}_{0+}^{\alpha ,\beta ;\psi }wd\xi  \notag \\
&=&\left\Vert b\right\Vert _{\infty }\zeta \chi\left( \zeta ^{\frac{1}{p-1}%
}c,\zeta ^{\frac{1}{p-1}}c\right) \mathfrak{M}_{2}\left( \int_{\Omega }\left\vert
^{\rm H}\mathfrak{D}_{0+}^{\alpha ,\beta ;\psi }\Psi_{2}\right\vert ^{p}d\xi\right) \text{ }\notag\\&&\times
\int_{\Omega }\left\vert ^{\rm H}\mathfrak{D}_{0+}^{\alpha ,\beta ;\psi }e_{q}\right\vert
^{q-2}\text{ }^{\rm H}\mathfrak{D}_{0+}^{\alpha ,\beta ;\psi }e_{q}\text{ }%
^{\rm H}\mathfrak{D}_{0+}^{\alpha ,\beta ;\psi }wd\xi  \notag \\
&\geq &\zeta \left\Vert b\right\Vert _{\infty }\chi\left( \zeta ^{\frac{1}{%
p-1}}c,\zeta ^{\frac{1}{p-1}}c\right) \int_{\Omega }wd\xi.\notag\\
\end{eqnarray}

Using the condition ${\bf(H_{5})}$, take $c$ large so that
\begin{equation}
\frac{1}{\left( \frac{\zeta \left\Vert b\right\Vert _{\infty }}{\mathfrak{M}_{2}}%
\right) ^{\frac{1}{q-1}}l_{q}}\geq \frac{\left[ \chi\left( \zeta ^{\frac{1}{%
p-1}}c,\zeta ^{\frac{1}{p-1}}c\right) \right] ^{\frac{1}{q-1}}}{c\zeta ^{%
\frac{1}{p-1}}}.
\end{equation}

So,
\begin{eqnarray}\label{3.9}
&&\chi\left( \zeta ^{\frac{1}{p-1}}c,\zeta ^{\frac{1}{p-1}}c\right)
\int_{\Omega }wd\xi \notag\\ &\geq &\int_{\Omega }\chi\left( \zeta ^{\frac{1}{p-1}}c,%
\left[ \chi\left( \zeta ^{\frac{1}{p-1}}c,\zeta ^{\frac{1}{p-1}}c\right) 
\right] ^{\frac{1}{q-1}}\left( \frac{\zeta \left\Vert b\right\Vert
_{\infty }}{\mathfrak{M}_{2}}\right) ^{\frac{1}{q-1}}l_{q}\right) wd\xi  \notag \\
&\geq &\int_{\Omega }\chi\left( c\zeta ^{\frac{1}{p-1}}\frac{e_{p}}{l_{p}},%
\left[ \chi\left( \zeta ^{\frac{1}{p-1}}c,\zeta ^{\frac{1}{p-1}}c\right) 
\right] ^{\frac{1}{q-1}}\left( \frac{\zeta \left\Vert b\right\Vert
_{\infty }}{\mathfrak{M}_{2}}\right) ^{\frac{1}{q-1}}e_{q}\right) wd\xi  \notag \\
&\geq &\int_{\Omega }\chi(\Psi_{1},\Psi_{2})wd\xi.
\end{eqnarray}

From (\ref{3.8}) and (\ref{3.9}), we deduce that
\begin{eqnarray}\label{3.10}
&&\mathfrak{M}_{2}\left( \int_{\Omega }\left\vert ^{\rm H}\mathfrak{D}_{0+}^{\alpha ,\beta ;\psi
}\Psi_{2}\right\vert ^{p}d\xi\right) \int_{\Omega }\left\vert ^{\rm H}\mathfrak{D}_{0+}^{\alpha
,\beta ;\psi }\Psi_{2}\right\vert ^{p-2}\text{ }^{\rm H}\mathfrak{D}_{0+}^{\alpha ,\beta ;\psi
}\Psi_{2}\text{ }^{\rm H}\mathfrak{D}_{0+}^{\alpha ,\beta ;\psi }wd\xi  \notag \\
&\geq &\zeta \left\Vert b\right\Vert _{\infty }\chi\left( \zeta ^{\frac{1}{%
p-1}}c,\zeta ^{\frac{1}{p-1}}c\right) \int_{\Omega }wd\xi  \notag \\
&\geq &\zeta \int_{\Omega }b(\xi)\chi(\Psi_{1},\Psi_{2})wd\xi.\notag\\
\end{eqnarray}

So from inequalities (\ref{3.7}) and (\ref{3.10}), it follows that $(\Psi_{1},\Psi_{2})$ is a supersolution of problem (\ref{P}). Moreover $\Psi_{i}\geq \Phi_{i}$ for $c$ large $i=1,2$. Therefore, by means of the {\bf Proposition \ref{Proposition2.2} }there exists a positive solution $(\eta_{1},\eta_{2})$ of (\ref{P}) with $\Phi_{1}\leq \eta_{1}\leq \Psi_{1}$ and $\Phi_{2}\leq \eta_{1}\leq \Psi_{2}$. This completes the proof of {\bf Theorem \ref{Teorema1.1}}.

\end{proof}

\section*{Acknowledgements}
{\color{red} The author thank very grateful to the anonymous reviewers for their useful comments that led to improvement of the manuscript.}

\section*{Funding}
Not applicable.

\section*{Availability of data and materials}
Not applicable.

\section*{{\bf Declarations}}

\section*{Ethical Approval}
Not applicable

\section*{Competing interests}

The author declare no competing interests.

\section*{Authors contributions}
 Not applicable



\end{document}